\documentclass[sn-mathphys-num]{sn-jnl}

\usepackage{graphicx}%
\usepackage{multirow}%
\usepackage{amsmath,amssymb,amsfonts}%
\usepackage{amsthm}%
\usepackage{mathrsfs}%
\usepackage[title]{appendix}%
\usepackage{xcolor}%
\usepackage{textcomp}%
\usepackage{manyfoot}%
\usepackage{booktabs}%
\usepackage{algorithm}%
\usepackage{algorithmicx}%
\usepackage{algpseudocode}%
\usepackage{listings}%

\theoremstyle{thmstyleone}%
\newtheorem{theorem}{Theorem}
\theoremstyle{thmstyletwo}%
\newtheorem{remark}{Remark}%
\newtheorem{lemma}{Lemma}%

\theoremstyle{thmstylethree}%

\begin{document}

\title[On the exactness of the Kiguradze--Kvinikadze blow-up condition]{On the exactness of the Kiguradze--Kvinikadze blow-up condition for nonlinear ordinary differential equations}








\author*[1]{\fnm{A. A.} \sur{Kon'kov}}\email{konkov@mech.math.msu.su}

\affil*[1]{\orgdiv{Department of Differential Equations}, \orgname{Faculty of Mechanics and Mathematics, Lo\-mo\-no\-sov Mo\-s\-cow State University}, \orgaddress{\street{Vorobyovy Gory}, \city{Moscow}, \postcode{119991}, \country{Russia}}}


\abstract{
We show the exactness of the Kiguradze--Kvinikadze blow-up condition for solutions of the Cauchy problem
$$
	w^{(m)}
	=
	f (t, w, \ldots, w^{(m - 1)}),
	\qquad
	w^{(i)} (0) = a_i \ge 0,
	\quad
	i = 0, \ldots, m - 1.
$$
}




\maketitle

\section{Introduction}\label{sec1}

We consider the differential equation
\begin{equation}
	w^{(m)}
	=
	f (t, w, \ldots, w^{(m - 1)}),
	\label{1.1}
\end{equation}
of order $m \ge 1$, where $f$ is a non-negative locally Caratheodory function~\cite{KK, KCbook}.
As is customary, a function $w : [0, T) \to {\mathbb R}$, $0 < T \le \infty$, is called a solution of~\eqref{1.1} on the interval $[0, T)$ if its derivatives $w^{(i)}$, $i = 0, \ldots, m - 1$, are locally absolutely continuous functions on $[0, T)$ and, moreover, equation~\eqref{1.1} holds for almost all $t \in [0, T)$. 
By a global solution of~\eqref{1.1} we mean a solution of this equation on the whole interval $[0, \infty)$. 
We also assume that solutions of~\eqref{1.1} satisfy the initial conditions
\begin{equation}
	w^{(i)} (0) = a_i \ge 0,
	\quad
	i = 0, \ldots, m - 1.
	\label{1.2}
\end{equation}

In paper~\cite[Theorem~1.2]{KK}, under the assumption that 
\begin{equation}
	f (t, x_0, \ldots, x_{m - 1})
	\ge
	q (t)
	h (x_k)
	\label{1.3}
\end{equation}
for almost all $t \in [0, \infty)$ and for all $x_i \in [0, \infty)$, $i = 0, \ldots, m - 1$, where $0 \le k \le m - 1$ is an integer, ${q : [0, \infty) \to [0, \infty)}$ is a locally integrable function, and ${h : [0, \infty) \to [0, \infty)}$ is a non-decreasing continuous function positive on the interval $(0, \infty)$,
it was proved that the condition
\begin{equation}
	\int_1^\infty
	h^{
		- 1 / (m - k)
	}
	(s)
	s^{
		1 / (m - k) - 1
	}
	ds
	<
	\infty
	\label{1.4}
\end{equation}
implies the absence of global solutions of problem~\eqref{1.1}, \eqref{1.2} for all sufficiently large initial values.
In so doing, 
it was required that
\begin{equation}
	\operatorname{mes} \{ \tau > t : q (\tau) > 0 \} > 0 
	\label{1.5}
\end{equation}
for all $t \ge 0$.
Replacing~\eqref{1.5} by a stronger condition, one can show that~\eqref{1.4} guarantees the absence of solutions of problem~\eqref{1.1}, \eqref{1.2} for all positive initial values too~\cite{meIzv, meJMAA, meJMS}.

For power-law nonlinearities $h (s) = s^\lambda$, condition~\eqref{1.4} takes the simple form $\lambda > 1$. 
The case of such nonlinearities are
studied in sufficient detail~\cite{KCbook, Ast2019, Ast2017, Ast2014, CDKM, I, IR, K, MOG_2008, MOG_2005}.
In particular, it is well-known that if the reverse inequality holds in~\eqref{1.3} with $h (s) = s^\lambda$, where $0 \le \lambda \le 1$ and $q$ is a non-negative locally integrable function, then problem~\eqref{1.1}, \eqref{1.2} has a global solution for all non-negative initial values.
It seems plausible that a similar statement should be true for the general nonlinearity. 
Theorem~\ref{t2.1} and~\ref{t2.2} proved in our paper testifies to the validity of this hypothesis.

\section{Main results}

\begin{theorem}\label{t2.1}
Let the right-hand side of equation~\eqref{1.1} satisfy the inequality
\begin{equation}
	f (t, x_0, \ldots, x_{m - 1})
	\le
	q (t)
	h (x_k)
	\label{t2.1.1}
\end{equation}
for almost all $t \in [0, \infty)$ and for all $x_i \in [0, \infty)$, $i = 0, \ldots, m - 1$, 
where $0 \le k \le m - 1$ is an integer, ${q : [0, \infty) \to [0, \infty)}$ is a locally integrable function, and ${h : [0, \infty) \to [0, \infty)}$ is a non-decreasing continuous function 
such that $h (s) > 0$ for all $s \in (0, \infty)$.
If 
\begin{equation}
	\int_1^\infty
	h^{
		- 1 / (m - k)
	}
	(s)
	s^{
		1 / (m - k) - 1
	}
	ds
	=
	\infty,
	\label{t2.1.2}
\end{equation}
then any solution of problem~\eqref{1.1}, \eqref{1.2} can be extended to the whole interval $[0, \infty)$.
\end{theorem}

\begin{remark}\label{r2.1}
From~\eqref{t2.1.2}, it obviously follows that
$$
	\int_1^\infty
	h^{
		- 1 / (m - k)
	}
	(\alpha \zeta)
	\zeta^{
		1 / (m - k) - 1
	}
	d \zeta
	=
	\infty
$$
for any real number $\alpha > 0$. To verify this, it is sufficient to make the change of variables $s = \alpha \zeta$ under the integral sign in the left-hand side of~\eqref{t2.1.2}.
\end{remark}

\begin{theorem}\label{t2.2}
Let the conditions of Theorem~\ref{t2.1} be fulfilled. Then problem~\eqref{1.1}, \eqref{1.2} has a global solution for all initial values.
\end{theorem}

To prove Theorems~\ref{t2.1} and~\ref{t2.2}, we need the following statements.

\begin{lemma}\label{l2.2}
Let $u$ be a positive solution of the equation
\begin{equation}
	\frac{
		d
	}{
		d t
	}
	u^{1 / n} (t)
	= 
	g^{1 / n} (u)
	\label{l2.2.1}
\end{equation}
on the interval $[0, T)$, 
where $n \ge 1$ is an integer and $g : (0, \infty) \to (0, \infty)$  is a non-decreasing continuous function.
Then there exist real numbers $\alpha > 0$ and $\beta > 0$ depending only on $n$ such that
\begin{equation}
	u (t) - u (0)
	\ge
	\alpha
	\int_0^t
	(t - \tau)^{n - 1}
	g (\beta u)
	d \tau
	\label{l2.2.2}
\end{equation}
for all $t \in [0, T)$.
\end{lemma}

\begin{proof}
Integrating~\eqref{l2.2.1}, we have
$$
	u (t) 
	=
	\left(
		u^{1 / n} (0)
		+
		\int_0^t
		g^{1 / n} (u)
		d \tau
	\right)^n
$$
for all $t \in [0, \infty)$, whence it follows that
\begin{equation}
	u (t) - u (t_1)
	\ge
	\left(
		\int_{t_1}^{t}
		g^{1 / n} (u)
		d \tau
	\right)^n
	\ge
	(t - t_1)^n
	g (u (t_1))
	\label{pl2.2.1}
\end{equation}
for all real numbers $0 \le t_1 \le t$.
Put
$$
	\alpha 
	= 
	\frac{
		1
	}{
		2^{n + 1}
	} 
	\quad
	\mbox{and}
	\quad
	\beta 
	= 
	\frac{
		1
	}{
		1 + 2^n
	}.
$$

Let $0 < t < T$ be some given real number and $l \ge 1$ be the minimal integer such that $u (0) \ge \beta^l u (t)$.
We prove~\eqref{l2.2.2} by induction on $l$.
At fist, assume that $l = 1$. It is obvious that
$$
	\int_{t / 2}^t
	(t - \tau)^{n - 1}
	g (\beta u)
	d \tau
	\le
	\frac{1}{n}
	\left(
		\frac{t}{2}
	\right)^n
	g (\beta u (t)).
$$
In so doing,~\eqref{pl2.2.1} implies the estimate
$$
	u (t) - u (t / 2)
	\ge
	\left(
		\frac{t}{2}
	\right)^n
	g (u (t / 2))
	\ge
	\left(
		\frac{t}{2}
	\right)^n
	g (\beta u (t)).
$$
Hence,
\begin{equation}
	u (t) - u (t / 2)
	\ge
	\int_{t / 2}^t
	(t - \tau)^{n - 1}
	g (\beta u)
	d \tau.
	\label{pl2.2.2}
\end{equation}
On the other hand, from the monotonicity of the function $g$, it follows that
$$
	\int_{t / 2}^t
	(t - \tau)^{n - 1}
	g (\beta u)
	d \tau
	\ge
	\frac{1}{n}
	\left(
		\frac{t}{2}
	\right)^n
	g (\beta u (t / 2))
	\ge
	\frac{1}{2^n}
	\int_0^{t / 2}
	(t - \tau)^{n - 1}
	g (\beta u)
	d \tau;
$$
therefore,
$$
	\int_{t / 2}^t
	(t - \tau)^{n - 1}
	g (\beta u)
	d \tau
	\ge
	\frac{
		1
	}{
		1 + 2^n
	}
	\int_0^t
	(t - \tau)^{n - 1}
	g (\beta u)
	d \tau.
$$
Combining this with~\eqref{pl2.2.2}, we obtain
\begin{equation}
	u (t) - u (t / 2)
	\ge
	\frac{
		1
	}{
		1 + 2^n
	}
	\int_0^t
	(t - \tau)^{n - 1}
	g (\beta u)
	d \tau,
	\label{PT2.3.3}
\end{equation}
whence~\eqref{l2.2.2} follows at once.

Now, assume that~\eqref{l2.2.2} is valid for some $l = l_0$, where $l_0$ is a positive integer. Let us show that this inequality is also valid for $l = l_0 + 1$. Take a real number $0 < t_1 < t$ satisfying the condition $u (t_1) = \beta u (t)$.
If $2 t_1 \le t$, then repeating the previous reasoning, we again arrive at~\eqref{PT2.3.3} and, consequently, at~\eqref{l2.2.2}. Consider the case of $2 t_1 > t$. By the induction hypothesis, we have
\begin{equation}
	u (t_1) - u (0)
	\ge
	\alpha
	\int_0^{t_1}
	(t_1 - \tau)^{n - 1}
	g (\beta u)
	d \tau.
	\label{pl2.2.4}
\end{equation}
Let us put $t_* = 2 t_1 - t$.
Since $2 (t_1 - \tau) \ge t - \tau$ for all $\tau \in (0, t_*)$, one can assert that
$$
	\int_0^{t_1}
	(t_1 - \tau)^{n - 1}
	g (\beta u)
	d \tau
	\ge
	\frac{
		1
	}{
		2^{n - 1}
	}
	\int_0^{t_*}
	(t - \tau)^{n - 1}
	g (\beta u)
	d \tau;
$$
therefore,~\eqref{pl2.2.4} implies the estimate
$$
	u (t_1)
	\ge
	\frac{
		\alpha
	}{
		2^{n - 1}
	}
	\int_0^{t_*}
	(t - \tau)^{n - 1}
	g (\beta u)
	d \tau.
$$
In view of the equality $u (t) - u (t_1) = (1 / \beta - 1) u (t_1)$, this yields
\begin{align}
	u (t) - u (t_1)
	&
	\ge
	\frac{
		\alpha
	}{
		2^{n - 1}
	}
	\left(
		\frac{1}{\beta}
		-
		1
	\right)
	\int_0^{t_*}
	(t - \tau)^{n - 1}
	g (\beta u)
	d \tau
	\nonumber
	\\
	&
	=
	2
	\alpha
	\int_0^{t_*}
	(t - \tau)^{n - 1}
	g (\beta u)
	d \tau.
	\label{PT2.3.5}
\end{align}
From~\eqref{pl2.2.1}, it follows that
$$
	u (t) - u (t_1)
	\ge
	(t - t_1)^n
	g (u (t_1))
	=
	(t - t_1)^n
	g (\beta u (t)).
$$
At the same time,
$$
	\int_{t_*}^t
	(t - \tau)^{n - 1}
	g (\beta u)
	d \tau
	\le
	(t - t_*)^n
	g (\beta u (t))
	=
	2^n
	(t - t_1)^n
	g (\beta u (t)).
$$
Thus, one can assert that
$$
	u (t) - u (t_1)
	\ge
	\frac{1}{2^n}
	\int_{t_*}^t
	(t - \tau)^{n - 1}
	g (\beta u)
	d \tau
	=
	2 
	\alpha
	\int_{t_*}^t
	(t - \tau)^{n - 1}
	g (\beta u)
	d \tau.
$$
Summing the last expression with~\eqref{PT2.3.5}, we arrive at the estimate
$$
	u (t) - u (t_1)
	\ge
	\alpha
	\int_0^t
	(t - \tau)^{n - 1}
	g (\beta u)
	d \tau,
$$
whence~\eqref{l2.2.2} obviously follows.
\end{proof}

\begin{lemma}\label{l2.3}
Let $h : (0, \infty) \to (0, \infty)$ be a non-decreasing continuous function such that
\begin{equation}
	\int_1^\infty
	h^{
		- 1 / n
	}
	(s)
	s^{
		1 / n - 1
	}
	ds
	=
	\infty,
	\label{l2.3.1}
\end{equation}
where $n \ge 1$ is an integer.
Then the Cauchy problem
\begin{equation}
	v^{(n)} = h (v),
	\qquad
	v^{(i)} (0) = b_i > 0,
	\quad
	i = 0, \ldots, n - 1,
	\label{l2.3.2}
\end{equation}
has a global solution for all initial values.
\end{lemma}

\begin{proof}
We put
$$
	v_0 (t)
	=
	\sum_{i = 0}^{n - 1}
	\frac{
		b_i 
		t^i
	}{
		i!
	}
$$
and
\begin{equation}
	v_j (t)
	=
	\sum_{i = 0}^{n - 1}
	\frac{
		b_i 
		t^i
	}{
		i!
	}
	+
	\frac{
		1
	}{
		(n - 1)!
	}
	\int_0^t
	(t - \tau)^{n - 1}
	h (v_{j - 1})
	d \tau,
	\quad
	j = 1,2,\ldots.
	\label{pl2.3.1}
\end{equation}
It is easy to see that 
$$
	0 \le v_{j - 1} (t) \le v_j (t)
$$
for all $t \in [0, \infty)$, $j = 1, 2, \ldots$. 
Further, let $T > 0$ be some real number and let
\begin{equation}
	g (s)
	=
	\frac{
		1
	}{
		\alpha
		(n - 1)!
	}
	h 
	\left(
		\frac{s}{\beta}
	\right),
	\quad
	s \in (0, \infty),
	\label{pl2.3.2}
\end{equation}
where $\alpha > 0$ and $\beta > 0$ are the constants in Lemma~\ref{l2.2}.
We show that equation~\eqref{l2.2.1} has a global solution satisfying the initial condition 
\begin{equation}
	u (0) 
	= 
	\sum_{i = 0}^{n - 1}
	\frac{
		b_i 
		T^i
	}{
		i!
	}.
	\label{pl2.3.3}
\end{equation}
Indeed, let us denote $w = u^{1/n}$. Then~\eqref{l2.2.1} takes the form
$$
	\frac{
		d
	}{
		d t
	}
	w (t)
	= 
	g^{1 / n} (w^n).
$$
Integrating this equation, we obtain
$$
	\int_{
		w (0)
	}^{
		w (t)
	}
	g^{- 1 / n} (\zeta^n)
	d \zeta
	=
	t.
$$
After the changing of variables $s = \zeta^n$, this implies that
$$
	\frac{1}{n}
	\int_{
		u (0)
	}^{
		u (t)
	}
	g^{- 1 / n} (s)
	s^{1 / n - 1}
	ds
	=
	t
$$
for all $t \in (0, \infty)$. 
In view of condition~\eqref{l2.3.1} and Remark~\ref{r2.1}, the integral in the right-hand side of the last equality tends to infinity as $u (t) \to \infty$, which obviously guarantees that a global solution of~\eqref{l2.2.1}, \eqref{pl2.3.3} exists.
By Lemma~\ref{l2.2}, this solution satisfies inequality~\eqref{l2.2.2}, whence in accordance with~\eqref{pl2.3.2} and~\eqref{pl2.3.3} we have
$$
	u (t)
	\ge
	\sum_{i = 0}^{n - 1}
	\frac{
		b_i 
		t^i
	}{
		i!
	}
	+
	\frac{
		1
	}{
		(n - 1)!
	}
	\int_0^t
	(t - \tau)^{n - 1}
	h (u)
	d \tau
$$
for all $t \in [0, T]$.
Since $h$ is a positive non-decreasing function, one can verify by induction that
$$
	v_j (t) \le u (t)
$$
for all $t \in [0, T]$, $j = 0,1,2,\ldots$. Thus, there is a function $v \in L_1 ([0, T])$ such that $v_j (t) \to v (t)$ as $j \to \infty$ for all $t \in [0, T]$.
Passing to the limit as $j \to \infty$ in~\eqref{pl2.3.1},
by Lebesgue's dominated convergence theorem,
we obtain
$$
	v (t)
	=
	\sum_{i = 0}^{n - 1}
	\frac{
		b_i 
		t^i
	}{
		i!
	}
	+
	\frac{
		1
	}{
		(n - 1)!
	}
	\int_0^t
	(t - \tau)^{n - 1}
	h (v)
	d \tau
$$
for all $t \in [0, T]$.
Since $T > 0$ can be an arbitrary real number, the last equality is valid for all $t \in [0, \infty)$. 
In so doing, the limit function $v$ can obviously be defined on the whole interval $[0, \infty)$. Finally, to complete the proof, it remains to verify by direct differentiation that $v$ is a solution of~\eqref{l2.3.2}.
\end{proof}

\begin{lemma}\label{l2.4}
Let the conditions of Theorem~\ref{t2.1} be fulfilled.
Then any solution of the Cauchy problem
\begin{equation}
	v^{(m)} = q (t) h (v^{(k)}),
	\qquad
	v^{(i)} (0) = a_i \ge 0,
	\quad
	i = 0, \ldots, m - 1,
	\label{l2.1.1}
\end{equation}
can be extended to the whole interval $[0, \infty)$.
\end{lemma}

\begin{proof}
Let $v$ be a solution of~\eqref{l2.1.1} and, moreover, $[0, T)$ be the maximal interval to which this solution can be extended. 
Assume by contradiction that $T < \infty$.
We denote $u = v^{(k)}$ and $n = m - k$. It is easy to see that $u$ is a solution of the Cauchy problem
\begin{equation}
	u^{(n)} = q (t) h (u),
	\qquad
	u^{(i)} (0) = a_{k + i},
	\quad
	i = 0, \ldots, n - 1,
	\label{pl2.4.1}
\end{equation}
on the interval $[0, T)$. In so doing, the relation
\begin{equation}
	\lim_{t \to T - 0}
	u (t)
	=
	\infty
	\label{pl2.4.2}
\end{equation}
must be satisfied; otherwise, equation~\eqref{l2.1.1} implies that
$$
	\lim_{t \to T - 0}
	v^{(i)} (t)
	<
	\infty,
	\quad
	i = 0, \ldots, m - 1.
$$
Therefore, $v$ can be extended to the interval $[T, T + \varepsilon)$ for some $\varepsilon > 0$ by a solution of the Cauchy problem
$$
	v^{(n)} = q (t) h (v^{(k)}),
	\qquad
	v^{(i)} (T) 
	=
	\lim_{t \to T - 0}
	v^{(i)} (t),
	\quad
	i = 0, \ldots, m - 1.
$$
This contradicts the fact that $[0, T)$ is the maximal interval to which the solution $v$ can be extended.

Take a real number $t_0 \in [0, T)$ satisfying the condition $u (t_0) > 0$.
Let us consider an increasing sequence of real numbers $t_j \in (t_0, T)$ such that $u (t_j) = 2^j u (t_0)$, $j = 1, 2, \ldots$.
We also put $\tau_0 = 0$ and 
\begin{equation}
	\tau_j 
	= 
	\tau_{j - 1}
	+
	t_j
	-
	t_{j - 1} 
	+
	\varepsilon_j,
	\quad
	j = 1, 2, \ldots,
	\label{pl2.4.3}
\end{equation}
where
$$
	\varepsilon_j
	=
	\int_{
		t_{j - 1}
	}^{
		t_j
	}
	q (t)
	dt
	\ge
	0.
$$
Since $q$ is a locally integrable function, it is obvious that
\begin{equation}
	\lim_{j \to \infty}
	\tau_j
	<
	\infty.
	\label{pl2.4.4}
\end{equation}

Further, let us take some real numbers
$
	b_i > a_{k + i},
	\;
	i = 0, \ldots, n - 1,
$ 
and let $w$ be a global solution of the Cauchy problem
\begin{equation}
	w^{(n)} = h (2 w),
	\qquad
	w^{(i)} (0) = b_i,
	\quad
	i = 0, \ldots, n - 1.
	\label{pl2.4.5}
\end{equation}
According to Lemma~\ref{l2.3} and Remark~\ref{r2.1}, such a solution exists.
We prove that
\begin{equation}
	w^{(i)} (\tau_j)
	\ge
	u^{(i)} (t_j),
	\quad
	i = 0, \ldots, n - 1,
	\label{pl2.4.6}
\end{equation}
for all $j = 0, 1, \ldots$.
If $j = 0$, then~\eqref{pl2.4.6} follows from the choice of the real numbers $b_i$. Now, assume that~\eqref{pl2.4.6} is valid for $j = j_0$, where $j_0 \ge 0$ is some integer. Let us show that this relation is also valid for $j = j_0 + 1$. Taking into account equation~\eqref{pl2.4.1}, we obtain
\begin{align*}
	u^{(i)} (t_{j + 1})
	=
	{}
	&
	\sum_{l = 0}^{n - i - 1}
	\frac{
		u^{(i + l)} (t_j) 
		(t_{j + 1} - t_j)^l
	}{
		l!
	}
	\\
	&
	{}
	+
	\frac{
		1
	}{
		(n - i - 1)!
	}
	\int_{
		t_j
	}^{
		t_{j + 1}
	}
	(t_{j + 1} - \tau)^{n - i - 1}
	q (\tau)
	h (u)
	d \tau,
	\quad
	i = 0, \ldots, n - 1.
\end{align*}
Since
\begin{align*}
	\int_{
		t_j
	}^{
		t_{j + 1}
	}
	(t_{j + 1} - \tau)^{n - i - 1}
	q (\tau)
	h (u)
	d \tau
	&
	{}
	\le
	(t_{j + 1} - t_j)^{n - i - 1}
	h (u(t_{j + 1}))
	\int_{
		t_j
	}^{
		t_{j + 1}
	}
	q (\tau)
	d \tau
	\\
	&
	{}
	=
	(t_{j + 1} - t_j)^{n - i - 1}
	h (u(t_{j + 1}))
	\varepsilon_{j + 1},
\end{align*}
this implies that
\begin{align}
	u^{(i)} (t_{j + 1})
	\le
	{}
	&
	\sum_{l = 0}^{n - i - 1}
	\frac{
		u^{(i + l)} (t_j) 
		(t_{j + 1} - t_j)^l
	}{
		l!
	}
	\nonumber
	\\
	&
	{}
	+
	\frac{
		(t_{j + 1} - t_j)^{n - i - 1}
		h (
			u (t_{j + 1})
		)
		\varepsilon_{j + 1}
	}{
		(n - i - 1)!
	},
	\quad
	i = 0, \ldots, n - 1.
	\label{pl2.4.7}
\end{align}
In its turn, equation~\eqref{pl2.4.5} yields
\begin{align*}
	w^{(i)} (\tau_{j + 1})
	=
	{}
	&
	\sum_{l = 0}^{n - i - 1}
	\frac{
		w^{(i + l)} (\tau_j) 
		(\tau_{j + 1} - \tau_j)^l
	}{
		l!
	}
	\\
	&
	{}
	+
	\frac{
		1
	}{
		(n - i - 1)!
	}
	\int_{
		\tau_j
	}^{
		\tau_{j + 1}
	}
	(\tau_{j + 1} - \tau)^{n - i - 1}
	h (2 w)
	d \tau,
	\quad
	i = 0, \ldots, n - 1,
\end{align*}
whence in accordance with the estimate
$$
	\int_{
		\tau_j
	}^{
		\tau_{j + 1}
	}
	(\tau_{j + 1} - \tau)^{n - i - 1}
	h (2 w)
	d \tau
	\ge
	\frac{
		(\tau_{j + 1} - \tau_j)^{n - i}
		h (2 w(\tau_j))
	}{
		n - i
	}
$$
it follows that
\begin{align}
	w^{(i)} (\tau_{j + 1})
	\ge
	{}
	&
	\sum_{l = 0}^{n - i - 1}
	\frac{
		w^{(i + l)} (\tau_j) 
		(\tau_{j + 1} - \tau_j)^l
	}{
		l!
	}
	\nonumber
	\\
	&
	{}
	+
	\frac{
		(\tau_{j + 1} - \tau_j)^{n - i}
		h (2 w(\tau_j))
	}{
		(n - i)!
	},
	\quad
	i = 0, \ldots, n - 1.
	\label{pl2.4.8}
\end{align}
By the induction hypothesis and the definition of the real number $t_{j + 1}$, we have
$$
	h (2 w (\tau_j)) \ge h (2 u (t_j)) = h (u (t_{j + 1})).
$$
Moreover, taking into account~\eqref{pl2.4.3}, one can assert that
$$
	\tau_{j + 1} - \tau_j
	\ge
	t_{j + 1} - t_j
$$
and
$$
	(\tau_{j + 1} - \tau_j)^{n - i}
	\ge
	(n - i)
	(t_{j + 1} - t_j)^{n - i - 1}
	\varepsilon_{j + 1},
	\quad
	i = 0, \ldots, n - 1;
$$
therefore,
$$
	\sum_{l = 0}^{n - i - 1}
	\frac{
		w^{(i + l)} (\tau_j) 
		(\tau_{j + 1} - \tau_j)^l
	}{
		l!
	}
	\ge
	\sum_{l = 0}^{n - i - 1}
	\frac{
		u^{(i + l)} (t_j) 
		(t_{j + 1} - t_j)^l
	}{
		l!
	}
$$
and
$$
	\frac{
		(\tau_{j + 1} - \tau_j)^{n - i}
		h (2 w (\tau_j))
	}{
		(n - i)!
	}
	\ge
	\frac{
		(t_{j + 1} - t_j)^{n - i - 1}
		h (
			u (t_{j + 1})
		)
		\varepsilon_{j + 1}
	}{
		(n - i - 1)!
	},
	\quad
	i = 0, \ldots, n - 1.
$$
Hence, combining~\eqref{pl2.4.7} and~\eqref{pl2.4.8}, we obtain
$$
	w^{(i)} (\tau_{j + 1})
	\ge
	u^{(i)} (t_{j + 1}),
	\quad
	i = 0, \ldots, n - 1.
$$

From~\eqref{pl2.4.4}, it immediately follows that
$$
	\lim_{j \to \infty}
	w (\tau_j)
	<
	\infty;
$$
therefore,~\eqref{pl2.4.6} implies the inequality
$$
	\lim_{j \to \infty}
	u (t_j)
	<
	\infty.
$$
This contradicts~\eqref{pl2.4.2}. The proof is completed.
\end{proof}

\begin{proof}[Proof of Theorem~\ref{t2.1}]
Let $w$ be a solution of problem~\eqref{1.1}, \eqref{1.2} and let $[0, T)$ be the maximal interval to which $w$ can be extended. 
Assume by contradiction that $T < \infty$.
We denote by ${\mathbb V}$ the set of solutions of problem~\eqref{l2.1.1} such that $\operatorname{dom} v \subset [0, T)$ and, moreover,
\begin{equation}
	w^{(i)} (t) \le v^{(i)} (t),
	\quad
	i = 0, \ldots, m - 1,
	\label{pt2.1.1}
\end{equation}
for all $t \in \operatorname{dom} v$.
As is customary, by $\operatorname{dom} v$ we mean the domain of a solution $v$. 
Thus, for any $v \in {\mathbb V}$ we have $\operatorname{dom} v = [0, t_v)$, where $0 < t_v \le T$ is some real number.

It can be seen that ${\mathbb V} \ne \emptyset$. 
Indeed, integrating~\eqref{1.1}, \eqref{1.2}, we obtain
$$
	w (t)
	=
	\sum_{j = 0}^{m - 1}
	\frac{
		a_j
		t^j
	}{
		j!
	}
	+
	\frac{
		1
	}{
		(m - 1)!
	}
	\int_0^t
	(t - \tau)^{m - 1}
	f (\tau, w, \ldots, w^{(m - 1)})
	d \tau
$$
for all $t \in [0, T)$.
Let us put $v_0 = w$ and
\begin{equation}
	v_l (t)
	=
	\sum_{j = 0}^{m - 1}
	\frac{
		a_j
		t^j
	}{
		j!
	}
	+
	\frac{
		1
	}{
		(m - 1)!
	}
	\int_0^t
	(t - \tau)^{m - 1}
	q (\tau)
	h (v_{l-1}^{(k)})
	d \tau,
	\quad
	l = 1,2,\ldots.
	\label{pt2.1.2}
\end{equation}
By direct differentiation, it can be verified that
$$
	w^{(i)}(t)
	=
	\sum_{j = 0}^{m - i - 1}
	\frac{
		a_{i+j}
		t^j
	}{
		j!
	}
	+
	\frac{
		1
	}{
		(m - i - 1)!
	}
	\int_0^t
	(t - \tau)^{m - i - 1}
	f (\tau, w, \ldots, w^{(m - 1)})
	d \tau
$$
and
\begin{equation}
	v_l^{(i)} (t)
	=
	\sum_{j = 0}^{m - 1}
	\frac{
		a_{i + j}
		t^j
	}{
		j!
	}
	+
	\frac{
		1
	}{
		(m - i - 1)!
	}
	\int_0^t
	(t - \tau)^{m - i - 1}
	q (\tau)
	h (v_{l-1}^{(k)})
	d \tau,
	\;
	l = 1,2,\ldots,
	\label{pt2.1.3}
\end{equation}
for all $t \in [0, T)$, $i = 0, \ldots, m - 1$.

Taking into account~\eqref{t2.1.1} and the monotonicity of the function $h$, we have
\begin{equation}
	w^{(i)} (t) 
	\le
	v_{l-1}^{(i)} (t) 
	\le 
	v_l^{(i)} (t),
	\quad
	i = 0,\ldots, m - 1,
	\quad
	l = 1,2,\ldots,
	\label{pt2.1.4}
\end{equation}
for all $t \in [0, T)$.
At the same time, since $q$ is a locally integrable function, there exists $\varepsilon \in (0, T]$ such that $\{ v_l \}_{l=0}^\infty$ is a bounded sequence in the space $C^{m-1} (0, \varepsilon)$.
Therefore, the sequence $\{ v_l^{(k)} \}_{l=0}^\infty$ converges everywhere on $(0, \varepsilon)$ to some function from $L_\infty (0, \varepsilon)$.
In view of~\eqref{pt2.1.3}, this implies that $\{ v_l \}_{l=0}^\infty$ is a convergent sequence in the space $C^{m-1} (0, \varepsilon)$. Let us denote by $v$ the limit of this sequence.

In accordance with Lebesgue's dominated convergence theorem, passing in~\eqref{pt2.1.2} to the limit as $l \to \infty$, we obtain
$$
	v (t)
	=
	\sum_{j = 0}^{m - 1}
	\frac{
		a_j
		t^j
	}{
		j!
	}
	+
	\frac{
		1
	}{
		(m - 1)!
	}
	\int_0^t
	(t - \tau)^{m - 1}
	q (\tau)
	h (v^{(k)})
	d \tau
$$
for all $t \in [0, \varepsilon)$. Thus, the function $v$ is a solution of problem~\eqref{l2.1.1} on the interval $[0, \varepsilon)$. In view of~\eqref{pt2.1.4}, this function satisfies inequalities~\eqref{pt2.1.1} for all $t \in [0, \varepsilon)$.

Let us introduce a partial order relation on the set ${\mathbb V}$, assuming that $v_1 \preceq v_2$ if $\operatorname{dom} v_1 \subset \operatorname{dom} v_2$ and, in addition, $v_2$ coincides with $v_1$ on $\operatorname{dom} v_1$.
It is easy to see that any linearly ordered subset of ${\mathbb V}$ has a maximal element. Hence, according to Zorn's lemma, there exists a maximal element for the whole set ${\mathbb V}$. Let us denote this element by $v_*$ and its domain by $[0, t_*)$. 

Assume by contradiction that $t_* < T$. Since in accordance with Lemma~\ref{l2.4} any solution of problem~\eqref{l2.1.1} can be extended to the whole interval $[0, \infty)$, we have
\begin{equation}
	\lim_{t \to t_* - 0}
	v_*^{(i)} (t)
	<
	\infty,
	\quad
	i = 0, \ldots, m - 1.
	\label{pt2.1.6}
\end{equation}
Therefore, replacing relation~\eqref{pt2.1.2} by 
$$
	v_l (t)
	=
	\sum_{j = 0}^{m - 1}
	\frac{
		v_*^{(j)} (t_* - 0)
		(t - t_*)^j
	}{
		j!
	}
	+
	\frac{
		1
	}{
		(m - 1)!
	}
	\int_{t_*}^t
	(t - \tau)^{m - 1}
	q (\tau)
	h (v_{l-1}^{(k)})
	d \tau,
	\;
	l = 1,2,\ldots,
$$
and repeating the above argument,
we obtain for some sufficiently small real number $\varepsilon > 0$ a solution of the Cauchy problem
$$
	v^{(m)} = q (t) h (v^{(k)}),
	\qquad
	v^{(i)} (t_*) 
	= 
	\lim_{t \to t_* - 0}
	v_*^{(i)} (t),
	\quad
	i = 0, \ldots, m - 1,
$$
on the interval $[t_*, \varepsilon)$ satisfying inequalities~\eqref{pt2.1.1} for all $t \in [t_*, \varepsilon)$.
Extending $v_*$ by this solution to the interval $[0, t_* + \varepsilon)$, we arrive at a contradiction with the assumption that $v_*$ is the maximal element of the partial ordered set ${\mathbb V}$. 
Thus, $t_* = T$, whence in accordance with~\eqref{pt2.1.6} and the fact that
$$
	w^{(i)} (t) \le v_*^{(i)} (t),
	\quad
	i = 0, \ldots, m - 1,
$$
for all $t \in [0, t_*)$, we have
$$
	\lim_{t \to T - 0}
	w^{(i)} (t)
	<
	\infty,
	\quad
	i = 0, \ldots, m - 1.
$$
This, in turn, contradicts our assumption that $[0, T)$ is the maximal interval to which the solution $w$ can be extended.
\end{proof}

\begin{proof}[Proof of Theorem~\ref{t2.2}]
Since $f$ is a locally Caratheodory function, problem~\eqref{1.1}, \eqref{1.2} has a solution on the interval $[0, \varepsilon)$ for some real number $\varepsilon > 0$. By Theorem~\ref{t2.1}, this solution can be extended to the whole interval $[0, \infty)$.
\end{proof}

\bigskip

\noindent
{\bf Acknowledgments} 
This work
was supported by the Russian Ministry of Education and Science as part of the program of the Moscow Center for Fundamental and Applied Mathematics under the agreement 075-15-2022-284  (critical exponents),
and by the Russian Science Foundation, project 20-11-20272-$\Pi$ (estimates of global solutions).

\bigskip

\noindent
{\bf Data availability} Data sharing not applicable to this article as no datasets were generated or analyzed during the current study.


%
%

\end{document}